\documentclass[reqno]{amsart}[a4paper]
\usepackage{latexsym,amsfonts,amssymb,epsfig}
\usepackage{hyperref} 
\usepackage{color}

\DeclareMathOperator{\Aut}{Aut}

\DeclareMathOperator{\St}{St}     

\renewcommand{\a}{\alpha}
\renewcommand{\b}{\beta}
\newcommand{\g}{\gamma}

\newcommand{\Sym}{\mathrm {Sym}}      

\newtheorem{theorem}{Theorem} 
\newtheorem{corollary}[theorem]{Corollary}
\newtheorem{lemma}[theorem]{Lemma}
\theoremstyle{remark}
\newtheorem{remark}{\bf Remark}
\theoremstyle{example}

\theoremstyle{conjecture}

\begin{document}
\title{Dice periodic groups}
\author{Victor Petrogradsky}
\address{Department of Mathematics, University of Brasilia, 70910-900 Brasilia DF, Brazil}
\email{petrogradsky@rambler.ru}



\begin{abstract}
We construct a family of finitely generated infinite periodic groups.
The basic example is a 2-group, called the tetrahedron group. 
We generalize the construction by suggesting a family of infinite finitely generated dice groups.
We provide weak conditions under which dice groups are periodic,
where orders of elements are products involving finitely many given primes.
\end{abstract}

\maketitle


\section{Introduction}

Golod and Shafarevich gave first examples of finitely generated infinite $p$-groups for any prime $p$,
thus answering in negative to the General Burnside Problem. 
These examples are based on the construction of a family of finitely generated infinite dimensional associative nil-algebras~\cite{Golod64}.
Grigorchuk and Gupta-Sidki gave direct and elegant constructions of finitely generated $p$-groups~\cite{Grigorchuk80,GuptaSidki83}.

One of generalisations of the Grigorchuk and Gupta-Sidki groups are so called {\it constant spinal groups} that are subgroups of automorphisms of a regular rooted tree 
generated by rooted automorphisms and directed automorphisms whose action on a subtree is equal to global action, for details, conditions of periodicity, and more references see~\cite{Petschick23A}. 
An analogue of the Grigorchuk group, so called {\it Fibonacci restricted Lie algebra} having a nil $p$ mapping was constructed by the author~\cite{Pe06}.

In Section~1 we construct a basic example of an infinite 4-generated 2-group, called a {\it tetrahedron group}.
The proof of periodicity of this group resemble that of Gupta-Sidki group~\cite{GuptaSidki83}. 
This basic example was motivated by the second example of a self-similar Lie superalgebra suggested and studied earlier by the author in~\cite{Pe16}.
This example also coincides with the group $K_3$ constructed recently in~\cite{Petschick23B}. 
We study its properties and prove periodicity, which was established for a family of groups denoted as $K_r$, where $r>4$~\cite{Petschick23B}.
Periodicity was also proved for classes of spinal groups~\cite{Petschick23A}.  
Groups of small period growth acting on spherically homogeneous trees were recently constructed and studied in~\cite{Petschick23B,Petschick24}.

As a main result, 
in Section~2 we generalize the tetrahedron group by suggesting a family of infinite finitely generated {\it dice groups}.
We give a rather weak conditions on their periodicity, so we obtain infinite finitely generated  
periodic dice groups which periods are products  involving finitely many given primes (Theorem~\ref{Tmain}).
The family of dice groups was motivated by a family of restricted nil Lie algebras of extremely oscillating intermediate growth, called {\it Phoenix algebras}, constructed by the author
in a series of papers~\cite{Pe17,Pe20clover,Pe20flies}. 
The present research is close to examples developed in~\cite{Petschick23B,Petschick24} but goes in a different direction.
The proof of periodicity is close to that of the Gupta-Sidki group~\cite{GuptaSidki83}.

On groups acting on tress see~\cite{Grigorchuk00horizons,Nekr05}.
We denote $b^a:=a^{-1}ba$ in a group, $\mathbf C_n$ is the cyclic group of order $n$.
By $\langle S\rangle $ we denote the subgroup generated by a set $S$.
The actions are left.

\section{Basic example: Tetrahedron group}

More detailed definitions on groups acting on trees see in the next section.
Consider the alphabet $X:=\{0,1,2,3,4,5,6,7\}$, which elements we identify with the vertices of the 3-dimensional cube, as shown on Fig~\ref{Fig1}.
Let $a$ be the reflection of the cube determined by the first axis of space. 
Namely, consider the plane parallel to the left and right vertical planes cutting the cube into equal parts. Now $a$ is the reflection through this plane.
Similarly, denote by $b,c$ the reflections determined by the second and third axes of space. The reflections $a,b,c$ are shown as three green arrows on Fig~\ref{Fig1}.
Consider the group generated by these reflections $H:=\langle a,b,c\rangle\subset \Sym_8.$
Clearly, we obtain an abelian group of isometries of the cube and 
\begin{equation}\label{z23}
H\cong \langle a\rangle\times \langle b\rangle\times \langle c\rangle\cong   \mathbf C_2^3.
\end{equation}
Consider the tetrahedron with vertices marked red on Fig~\ref{Fig1}.
The subgroup of $H$ that leaves the red tetrahedron invariant is the Klein group $\bar H:=\{1,ab, ac, bc\}\cong \mathbf C_2^2$,
which three nontrivial elements are rotations at $180^\circ$ around lines connecting centers of the opposite faces.
The remaining vertices $\{1,2,4,7\}$ belong to another tetrahedron, which we call black.
Let $h\in H$, $k\in X$, then $h(k)={}^hk$ denotes the {\it left} action, which can be considered as a shift on the vector space (also elementary abelian group) $H=\mathbf C_2^3=X$ as well.

Let $T$ be the infinite rooted 8-ary regular tree, which vertices are identified with the language $X^*$.
The words of length $n$ are identified with $n$th level $V_n:=X^n$ of the tree, for all $n\ge 0$.
In particular, the root is considered of level zero and marked by $\emptyset$.
Let $v\in V_n\subset  T$ be a vertex of level $n$, $n\ge 0$. Then it has eight edges $vx$, $x\in X$ going down to vertices $V_{n+1}=X^{n+1}$ of $(n+1)$th level.  

We consider that $a,b,c\in \Aut T$ act on the first letter of nonempty words $v\in X^*\setminus\{\emptyset\}$, thus permuting the eight subtrees of $T$ starting from the first level.
Thus, they are {\it rooted automorphisms} of the tree $T$.  
Next, we define {\it directed} automorphism $w\in\Aut T$ recursively as
\begin{equation}\label{www}
w:=(w \rfloor_0, a \rfloor_3, b \rfloor_5, c \rfloor_6), 
\end{equation} 
where by $\rfloor$ we specify (here and below)   only the nontrivial sections at vertices $\{0,3,5,6\}$, which correspond to the red tetrahedron.
This recursion is shown by the red tetrahedron marked by respective sections on Fig~\ref{Fig1}.

Finally, we define the {\it tetrahedron group} $G:=\langle w,a,b,c\rangle\subset \Aut T$.
By our construction, $G$ is self-similar.
Clearly, the generators are involutions, i.e.
\begin{equation}\label{involution}
w^2=a^2=b^2=c^2=1.
\end{equation} 
\begin{figure}[h]
\centering
\caption{Cube vertices are identified with the set $X$. Element $w\in\Aut T$ is defined recursively~\eqref{www}, the nontrivial sections are placed at the vertices of the red tetrahedron and marked by respective letters.}
\label{Fig1} 
  \includegraphics[width=0.3\textwidth]{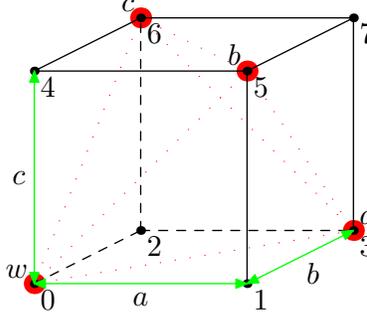}
\end{figure}

\begin{lemma} \label{L1}
$A:=\langle a,w \rangle \cong \mathbf D_8$, the dihedral group of order 8. 
\end{lemma}
\begin{proof} 
Observe that $w$, $w^a $ have sections $w$ at points $0,1$, respectively, and the latter cannot interact with sections $a,b,c$ inside $w$, $w^a$, because they are at vertices  of different parallel edges 23, 45, 67 of the cube,  see Fig.~\ref{Fig1}.
Hence, $w$, $w^a $ commute and $\langle w, w^a\rangle =\{1, w, w^a, ww^a\} \cong \mathbf C_2^2$.
By standard arguments $A$ is a product of two subgroups $\langle a\rangle$ and $\langle w, w^a\rangle$.  
Hence,
$$ A=\langle a,w \rangle=  \langle a \rangle \rightthreetimes \langle w, w^a\rangle  \cong\langle a\rangle  \rightthreetimes\mathbf C_2^2  \cong \langle a\rangle\wr \langle w\rangle  .$$
We have $(wa)^2=wawa=w w^a\ne 1$ and $(wa)^4=1$. 
Similarly, one checks that the elements of $A$ of order 4 are $wa$, $w^aa=aw$ because $ww^aa=waw$ is of order 2.
Hence, $A$ is not the quaternion group.
Clearly, $A$ is not abelian.  Therefore, $A\cong \mathbf D_8$.
\end{proof}
One obtains the dihedral group presentation (that is also easily obtained from the standard one):
\begin{equation}\label{Q8}
\mathbf D_8=\langle a,w \mid a^2=w^2=(wa)^4=1\rangle.
\end{equation}
\begin{lemma} $B:=\langle a,b,w \rangle\cong \mathbf C_2^2 \rightthreetimes  \mathbf D_8^2$,
where the action of $\mathbf C_2^2=\langle \bar a,\bar b\rangle $ is as follows:
$\bar a$ permutes two copies of the dihedral group, while $\bar b$ permutes the groups simultaneously flipping the involutions $a,w$~\eqref{Q8} as well.     
\end{lemma}
\begin{proof} 
Consider the action of $\langle a,b\rangle\cong \mathbf C_2^2$ on $w$ by conjugations. 
We get $\{w,w^a,w^b,w^{ab}\}$ having sections $w$ at the bottom face of the cube only,
namely at vertices  $\{0,1,2,3\}$, where they can interact with sections $a$ of the above conjugates as well. There appear section groups $\mathbf D_8$ at the bottom vertices by Lemma~\ref{L1}.
The remaining sections $b,c$ travel at the top face only, where they commute.
Thus, we get an embedding
$$B\subset  \langle a,b\rangle \rightthreetimes  \mathbf D_8^4\cong \mathbf C_2^2 \rightthreetimes  \mathbf D_8^4\cong \mathbf C_2^2 \wr \mathbf D_8.$$
Moreover,  observe that $w$, $w^{ab}$ have sections $w,a$ and $a,w$, respectively, at the opposite vertices 0,3. 
Hence, appearing section subgroups at 0,3 are related by the automorphism of $\mathbf D_8$
permuting generating involutions~\eqref{Q8}. On the other hand, the sections at 0,1 are independent.  
Finally, considering sections at 0,1 only we arrive at the claimed group. 
\end{proof}

\begin{corollary}
Any proper subset of the generating set $\{w,a,b,c\}$ of the group $G=\langle w,a,b,c\rangle$ generates a finite 2-group.
\end{corollary}

We have a bijection $X=\{ ^h0\mid h\in H\}\leftrightarrow  H$.
Namely, consider $h\in H$, we observe its action on the vertex 0 of the cube, set  $k:={}^h 0\in X$. Then we denote $h=:h_k$. 
Using Fig.~\ref{Fig1}, the following two lists are identified:
$$ H=\{h_0,h_1,h_2,h_3,h_4,h_5,h_6,h_7\}:=\{1,a,b,ab,c,ac,bc,abc\}.$$
Denote $$w_k:=w^{h_k}= h_k^{-1} w h_k, \qquad k\in X.$$ 
Observe that $w_k$ has the section $w$ at the cube vertex $k$, $k\in X$.
Namely, 
\begin{equation}\label{tetrapac}
w_k=(w \rfloor _{^h0}, a \rfloor_{^h3}, b \rfloor_{^h5}, c \rfloor_{^h6}),\qquad\text{where}\quad ^h0=k,\quad  k\in X,
\end{equation}
where these nontrivial sections are at vertices of one of two tetrahedrons and trivial sections are omitted.
\begin{lemma}\label{Lwkws}
Assume that $k,s\in X$ belong to different tetrahedrons. Then $w_k,w_s$ commute.
\end{lemma} 
\begin{proof}
Using~\eqref{tetrapac}, we see that the nontrivial sections of these elements belong to different tetrahedrons.
\end{proof}

\begin{lemma} \label{Lww}
Let $W:=\langle w_k,w_s \rangle$, $k,s\in X$. Then
\begin{enumerate}
\item $W$ is a finite 2-group.
\item for any $g\in W$ one has $g^4=1$.
\end{enumerate}
\end{lemma}
\begin{proof}
Case 1). Assume that $k,s$ belong to different tetrahedrons. Since  $w_k$, $w_s$ commute,  $\langle w_k,w_s \rangle\cong \mathbf C_2^2$.

Case 2). Assume that $k,s$ belong to the same tetrahedron. We have $^hk=s$ for a unique $h\in \bar H$. Under the shift $h$, the vertices of the tetrahedron split into two orbits of length two. 
We get two pairs  $(k,s)$ and, say, $(p,q)$.
Without loss of generality assume that $w_k$ has section $a$ at $s$.  Then we have
\begin{align*}
w_k&=(w \rfloor_k, a \rfloor_s, b \rfloor_p, c \rfloor_q),\\ 
w_s&=(a \rfloor_k, w \rfloor_s, c \rfloor_p, b \rfloor_q). 
\end{align*}   
We use Lemma~\ref{L1} and that $\langle b,c\rangle\cong \mathbf C_2^2$, and obtain
$$ \langle w_k,w_s\rangle \subset \Big(\mathbf D_8\big \rfloor\strut_k,\mathbf D_8\big \rfloor \strut_s,  \mathbf C_2^2\big \rfloor\strut_p ,\mathbf C_2^2\big \rfloor\strut_q \Big). $$
Then for any $g\in W$ we have $g^4=1$. 
\end{proof}

\begin{theorem}\label{Ttetra}
The tetrahedron group $G=\langle w,a,b,c\rangle\subset \Aut T$ is an infinite 2-group.
\end{theorem}
\begin{proof}
Consider $g\in G$, then $g=h_{j_0} w h_{j_1}\cdots w h_{j_s}$, where $h_{j_l}\in H$. 
By moving $h$-letters to the left, we get
\begin{equation} \label{present}
g= h w_{\g_1}w_{\g_2}\cdots w_{\g_m},\qquad \g_l\in X,\ h\in H, \quad m\ge 0,  
\end{equation} 
the neighbours $w_{\g_j}$  being different.
(By Lemma~\ref{Lwkws}, we can additionally collect $w_k$ for the red tetrahedron, and then for the black one. This observation might have some applications).
We define the {\it w-length} using presentation~\eqref{present} as 
$$L(g):= m,\qquad g\in G.$$
We draw attention that $h$ in~\eqref{present} can be equal to the identity. So, the elements of $H$ have zero $w$-length.

We prove periodicity of $g\in G$ by induction on $L(g)$.
The base of induction: let $L(g)=0$ then $g=h\in H$ and $g^2=1$ by~\eqref{z23}.

Fix $m\ge 1$ and assume that the claim is valid for all $g\in G$ with $L(g)< m$. 
Consider $g\in G$ with $L(g)=m$.   

{\bf Case 1.} Assume that  $h\ne 1$ in~\eqref{present}. 
Denote by $n_k$ the multiplicity of $w_k$ in~\eqref{present} for $k\in X$. Then
\begin{equation}
L(g)=\sum_{k=0}^7 n_k=m. 
\end{equation}
The action of the group $H=\mathbf C_2^3$ on $X$ by the shift $h\ne 1$ yields splitting into 4 orbits of length 2:  
\begin{align}
X&=\mathop{\cup}\limits_{j=1}^4\{\a_j,\b_j\},\qquad\text{where}\quad ^h\a_j=\b_j,\ j=1,\ldots,4;  \label{action}\\
L(g)&=\sum_{k=0}^7 n_k=\sum_{j=1}^4(n_{\a_j}+n_{\b_j})=m.\label{lgm1}
\end{align}
Using~\eqref{present}, we get
\begin{align}
g^2&=h w_{\g_1}\cdots w_{\g_m}\cdot  h w_{\g_1}\cdots w_{\g_m} = w_{\g_1}^h\cdots  w_{\g_m}^h w_{\g_1}\cdots w_{\g_m}\nonumber\\ 
&= w_{h(\g_1)}\cdots w_{h(\g_m)} w_{\g_1}\cdots w_{\g_m}\label{decomp00} \\
&=(g_0,g_1,\ldots,g_7)\in G^8\subset H\rightthreetimes G^8, \qquad g_k\in G. \label{decomp0} 
\end{align}
Since the action by $h$~\eqref{action} permutes $\a_j,\b_j$-components, where $j\in\{1,2,3,4\}$, the sections~\eqref{decomp0} of $g^2$ with indices $\a_j,\b_j$
have at most $n_{\a_j}+n_{\b_j}$ factors $w$ each. 
Hence, we can evaluate $w$-lengths of the sections $g_j$ of $g^2$ in~\eqref{decomp0}: 
\begin{equation}\label{lgb} 
\max\{L(g_{\a_j}), L(g_{\b_j})\}\le n_{\a_j}+n_{\b_j},\qquad j=1,2,3,4. 
\end{equation}

a) Assume that $n_{\a_j}+n_{\b_j}<m$ for all $j=1,2,3,4$ in~\eqref{lgm1}. Then using~\eqref{lgb} all sections of $g^2$ 
have $w$-lengths less than $m$ and we apply the inductive assumption. 

b) Assume that  $n_{\a_j}+n_{\b_j}=m$ for some $j\in \{1,2,3,4\}$ in~\eqref{lgm1}. 
Hence $n_{\a_i}+n_{\b_i}=0$ for the remaining indices $i\in \{1,2,3,4\}\setminus \{j\}$, this means that $w_{\a_j}, w_{\b_j}$ are the only letters appearing in~\eqref{decomp00}. 
Therefore, $g^2\in \langle w_{\a_j}, w_{\b_j} \rangle $  and the result follows by Lemma~\ref{Lww}.

{\bf Case 2.} Consider the case $h=1$ in presentation~\eqref{present}. We have
\begin{align}
g &=w_{\g_1}w_{\g_2}\cdots w_{\g_{m}}=(g_0,g_1,\ldots,g_7)\in G^8, \qquad g_k\in G; \label{decomp} \\
L(g)&=\sum_{k=0}^7 n_k=m,\qquad\text{where}\ n_k \text{ is multiplicity of\ } w_k, \ k\in X.\nonumber
\end{align}
Recall that section $w$ appears in each $w_k$ on $k$th place.
Hence, we get bounds on $w$-lengths of sections of~\eqref{decomp}:  
$$L(g_k)\le n_k, \qquad k\in X.$$ 

a) Assume that $n_k< m$ for all $k\in X$. Then $L(g_k)< m$ for all $k\in X$ and 
the claim follows by application of the inductive assumption to all sections.

b) Let for some $k\in X$ we have $n_k=m$. 
So, the decomposition~\eqref{decomp} consists only of $m$ factors $w_k$.  
Hence,  $m=1$ and  $g=w_k$. 
The inductive step follows by~\eqref{involution}.
Theorem is proved.
\end{proof}

\begin{remark}
We split the original unitary cube into 8 subcubes via cutting by planes parallel to the faces and passing through the middles of 4 parallel edges.
As was observed above, $a,b,c$ are isometries of the cube, they permute these 8 subcubes.
We repeat the process to the subcubes, etc.
We obtain an interpretation of the tetrahedron group $G$ as acting on the unitary cube similar to the original interpretation of the Grigorchuk group
acting on the unitary interval without binary points~\cite{Grigorchuk80}.
Now we exclude from the unitary cube the points one of which coordinates is binary. 
\end{remark}

\section{General construction: Dice groups and their periodicity}

\subsection{Groups acting on spherically homogeneous trees}
Below we mainly follow notations of~\cite{Grigorchuk00horizons}.
Fix a  {\it branching  sequence of alphabets} $\bar X:=(X_n\mid  n\ge 1)$, of respective sizes $m_{n}$, $n \ge 1$.
Let $n\ge 1$, introduce the set of specific words
$$\bar X^n:=X_1X_2\cdots X_n:=\{v=y_1\cdots y_{n}\mid  y_{i}\in X_i, 1\le i\le n\},$$
these words are called of {\it length} $|v|:=n$.
Denote by $\emptyset$ the word of zero length.
Put  $\bar X^*:=\cup_{n=0}^\infty \bar X^n$.
Let $\sigma$ be the {\it shift} operator which cuts the first member of the sequence, namely $\sigma \bar X:=(X_2,X_3,X_4,\ldots)$.

Consider the respective tree $T=T_{\bar X}$,
which set of {\it vertices} is $V(T):=\bar X^*$ and the set of {\it edges} $E(T)$ is
given by edges that connect all pairs of words $w$ and $wy$, where $w\in \bar X^n$ and $y\in X_{n+1}$.
Denote the set of words of length $n$ also as $V_n:=\bar X^n$, which we identify with the $n$th {\it level} of the tree.
We consider $\emptyset$ as the {\it root} of the tree, which constitutes the level $V_0$.
The tree $T_{\bar X}$ is called {\it spherically homogeneous}, it is determined by the branching sequence of alphabets $\bar X$.
The finite subtree of $T$ consisting of vertices of levels from 0 to $n$ along with respective connecting edges is denoted by $T_{[n]}$.
Consider $v\in V(T)$, then denote by $T_v$
the set of words in $V(T)$ with a prefix $v$, we identify $T_v$ with the subtree of $T$ with the root at $v$.
We have $T_u\cong T_v$, provided that $u,v$ belong to the same level, say, $V_n$, denote these subtrees as $T_{\langle n\rangle}$.
Clearly, $T_{\langle n\rangle}\cong T_{\sigma^n\bar X}$.

Denote by $\Aut T $ the group of automorphisms of the rooted tree $T$.
We will use the convention that automorphisms act on $T$ on the {\it left}.
Consider the natural mappings determined by restriction of the action:
$\pi_n:$ $\Aut T\twoheadrightarrow \Aut T_{[n]}\subset \Sym (V_n)$, $n\ge 1$.

Let $g\in \Aut T $. Fix $n\ge 1$, consider all vertices $v\in V_n$. Then one has decompositions
\begin{equation}\label{sections}
g(vw)=g(v) g|_{v}(w), \qquad  w\in T_{\langle n\rangle },\quad g(v)\in V_n, \quad g|_{v}\in\Aut T_{\langle n\rangle },
\end{equation}
where $g|_{v}\in\Aut T_{\langle n\rangle }$ are called {\it sections} (or {\it states}) of $g$ at $v$.
Also $g(v)=(\pi_n g)(v)$, $v\in V_n$, where $\pi_n g\in\Aut T_{[n]}$  is called the {\it top action} of $g$.
One obtains relations:
$$
f|_{vw}=(f|_{v})|_{w},\qquad (fg)|_{v}=f|_{g(v)}\cdot  g|_{v}; \qquad f,g\in\Aut T,\ v\in V_n, \ w\in T_{\langle n\rangle }.
$$
The pointwise stabilizer of the $k$th level of $T$ in a subgroup $G\subset \Aut T$ is denoted as $\St_G(k)$ and called $k$th {\it level stabilizer}.

Let a group $H$ acts from the {\it left} by permutations on a set $X$ and $G$ is a group,
we denote action as  $h(x)={^h x}$, where $h\in H$, $x\in X$.
Then the {\it (permutational) wreath product} $H\wr G$ is the semidirect product $H \rightthreetimes G^X$ where
$H$ acts by permutations of the direct factors. Assume that $X=\{x_1,\ldots,x_d\}$.
Then $G^X$ consists of all tuples $(g_1,\ldots,g_d)$, $g_i\in G$.
The multiplication of elements $h (g_1,\ldots,g_d)\in H\wr G$ is given by:
\begin{equation}\label{wr_prod}
\a(g_1,\ldots,g_d)\cdot \b(f_1,\ldots,f_d):=\a\b (g_{\b(1)} f_1,\ldots, g_{\b(d)} f_d),\qquad \a,\b\in H, \ f_i,g_j\in G.
\end{equation}
In a particular case, when $X=H$ we also get the {\it wreath product} $H\wr G=H \rightthreetimes G^H$.  

Denote $X_1=:\{x_1,\ldots,x_m\}$. Now,~\eqref{sections} and~\eqref{wr_prod} yield an isomorphism:
\begin{align*}
\psi_1: \Aut T & \cong \Sym_{X_1} \wr \Aut T_{\langle 1\rangle };\\
\psi_1:g  & \mapsto \pi_1(g)\big(g\big|\strut_{x_1},\ldots,g\big|\strut_{x_m} \big),\qquad g\in \Aut T.
\end{align*}
Denote by $V_n=:\{y_{1},\ldots,y_{M_n}\}$ the vertices of the $n$th level.
More generally, \eqref{sections} yields homomorphisms:
\begin{equation}\label{selfN}
\begin{split}
\psi_n: \Aut T &\cong \Sym_{X_1}\wr \cdots \wr \Sym_{X_n} \wr \Aut T_{\langle n\rangle };\\
g  &\mapsto \pi_n(g) \big(g|\strut_{y_1},\ldots,g|\strut_{{y}_{M_n}} \big),\qquad g\in \Aut T.
\end{split}
\end{equation}

Now consider a particular case. We suppose that the alphabets above are the same, namely $X_i=X$ for all $i\ge 1$. 
The we get a {\it regular rooted tree} $T=T_{\bar X}$.
In this case a group $G\subset\Aut T$ is called {\it self-similar} provided that all sections of the first level belong to $G$.
Namely, for any $g\in G$, $x\in X$ we consider the section of the first level $g|_x\in \Aut T_{\langle 1\rangle}\cong \Aut T$. Then $g|_x\in G$. 

\subsection{Dice groups}
Let $P$ be a finite set of primes.
Let $\bar N:=(N_1,N_2,N_3,\ldots)$ be an infinite tuple of positive integers.
Consider an infinite series of elementary abelian groups $H_i:=\mathbf C_{p_i}^{N_i}$, where $p_i$ is a prime from $P$,
these groups are identified with the sets $H_i$, which we refer to as {\it cubes}. 
The identity elements are denoted as $1_i\in H_i$. Denote by $A_i:=\{a_{ij}\mid 1\le j\le N_i\}$ the standard basis of the group $H_i$.  
Now we get a spherically homogeneous tree $T$  determined by the sequence $(H_1,H_2,H_3,\ldots)$, see above.  
Observe that each group $H_i$ acts by rooted automorphisms on all subtrees $T_{\langle i-1 \rangle}$ of level $i-1$, for all $i\ge 1$.

We {\it roll a dice} infinitely many times, namely we choose $N_{i+1}$ nonidentity {\it defining points} in the cube~$H_i$:
\begin{equation}\label{Yi}
Y_{i}:=\{y_{ij}\mid 1\le j\le N_{i+1}\}\subset H_i\setminus \{1_{i}\}\quad  \text{for all}\quad i\ge 1.
\end{equation}
By construction, these chosen points $Y_i$ are in a bijective correspondence with the basis $A_{i+1}$ of the next group $H_{i+1}$.
By putting elements of $A_{i+1}$ in sections labelled by the respective elements of $Y_i$\eqref{Yi}, we determine recursively a sequence of directed automorphisms of the subtrees
\begin{equation}\label{wi}
w_i:=\Big(w_{i+1}\Big \rfloor_{1_{i}}, a_{i+1,j}\Big \rfloor_{y_{ij}\in Y_{i} },  
      1_{i+1}\Big \rfloor_{y\in H_i\setminus  (Y_i\cup \{1_{i}\}) }  \Big)\in \Aut T_{\langle i-1\rangle },\qquad i\ge 1.
\end{equation}

Now we define a family of finitely generated {\it dice groups} determined by the dice rolling:
\begin{equation*}
G_i:=\langle w_i, A_i \rangle\subset \Aut T_{\langle i-1\rangle },\qquad i\ge 1.
\end{equation*}
In particular, we get our main object, the {\it dice group} $G:=\langle w_1, A_1 \rangle\subset \Aut T. $
Observe that
\begin{align}
(a_{ij})^{p_i}&=1_i,\qquad 1\le j\le N_i,\ i\ge 1;\nonumber\\
(w_i)^{q}&=1, \quad \text{where}\ q:=\prod_{p\in P} p,\quad i\ge 1.\label{qqq}
\end{align}

\subsection{Lucky dice roll at step $i$}
We start with a general condition providing particular cases afterwards.
The most clear particular version is the last one  {\bf DDmin}.
\begin{itemize}
\item{\bf D:} Fix $i\ge 1$. We say that we made a {\it lucky dice roll at step} $i$ provided the following is valid.\\
Consider any line $l\subset H_i$ passing through $1_i\in H_i$ and proceed recursively as follows
\begin{enumerate}
\item As the initial step, denote $m:=i$ and start with the defining points on the line, namely set 
\begin{equation}\label{line}
Z_m:=Y_m\cap l.
\end{equation}
\item If $Z_{m}=\emptyset$ then the process stops. Otherwise we go to the next step.
\item Denote 
$Z_m=:\{y_{m,j_1},\ldots, {y_{m, j_{S_m}}}\}\subset H_{m}$, so $|Z_m|=S_m\ge 1.$ 
Consider $S_m$-dimensional face 
$\Pi_{m+1}:=\langle a_{m+1,j_1}, \ldots, a_{m+1,j_{S_m}}\rangle$ of the next cube $H_{m+1}$
passing through $1_{m+1}$ and parallel to the respective vectors $\{a_{m+1,j_1}, \ldots, a_{m+1,j_{S_m}}\}\subset A_{m+1}$. 
Consider the defining points in this face:
$$Z_{m+1}:=Y_{m+1}\cap \Pi_{m+1}. $$
\item we return to step ii) with $m+1$ and $Z_{m+1}$.
\end{enumerate}
We assume that for any such line $l\subset H_i$ this process terminates.
\end{itemize}
A particular case is that the process always stops after two steps, namely $Z_{i+1}=\emptyset$. Thus we get a subcase.
\begin{itemize}
\item {\bf DD:} Fix $i\ge 1$. We say that we made a {\it lucky dice roll at step} $i$ provided the following is valid. \\
Consider any line $l\subset H_i$ passing through $1_i\in H_i$ and assume that $Y_i\cap l=  \{y_{i,j_1},\ldots, {y_{i, j_S}}\}$. 
Consider $S$-dimensional face $\Pi :=\langle a_{i+1,j_1}, \ldots, a_{i+1,j_{S}}\rangle$ of the next cube $H_{i+1}$
passing through $1_{i+1}$ and parallel to the respective vectors $\{a_{i+1,j_1}, \ldots, a_{i+1,j_S}\}$.
We require that
$$   Y_{i+1}\cap \Pi=\emptyset,\qquad \text{for all such lines } l\subset H_i.$$
\end{itemize}
More {\bf specific cases} of the {\it lucky dice roll at step} $i$ of type {\bf DD} are:
\begin{itemize}
\item {\bf DDmax:} Fix $i\ge 1$. Set
$$
d_i:=\max\{ |l\cap  Y_i| \mid \text{for all lines } l\subset H_i  \text{ passing through } 1_i\}.
$$
Then we require that all $d_i$-dimensional faces of the next cube $H_{i+1}$ passing through $1_{i+1}$ have no points of $Y_{i+1}$.
\end{itemize}
Since $H_i=\mathbf C_{p_i}^{N_i}$, we have $d_i\le p_i-1$ above. So, we get a simpler version.
\begin{itemize}
\item {\bf DDmax-1:} Fix $i\ge 1$.  We require that all $(p_i{-}1)$-dimensional faces of the next cube $H_{i+1}$ passing through $1_{i+1}$
have no points of $Y_{i+1}$.
\end{itemize}
By setting $d_i=1$, we get probably the easiest particular case of a {\it lucky dice roll at step} $i$: 
\begin{itemize}
\item {\bf DDmin:} Fix $i\ge 1$. We require that
\begin{itemize}
\item all lines in $H_i$ passing through $1_i$ contain at most one point of $Y_i$ and
\item all axes of $H_{i+1}$ (i.e. edges passing through $1_{i+1}$)  
have no  points of $Y_{i+1}$.
\end{itemize}
\end{itemize}

\subsection{Periodicity of the dice groups}

Now we formulate and prove our main result that yields a wide class of infinite finitely generated periodic groups, periods of which elements are products involving finitely many given primes.
\begin{theorem}\label{Tmain}
Let $P$ be a finite set of primes,
$\bar N:=(N_1,N_2,N_3,\ldots)$ a tuple of positive integers. 
Consider the series of elementary abelian groups $H_i:=\mathbf C_{p_i}^{N_i}$, where $p_i\in P$ for all $i\ge 1$.
Let $T=T_{\bar H}$ be the spherically homogeneous tree determined by the sequence of alphabets $\bar H=(H_1,H_2,H_3,\ldots)$. 
Let $A_i:=\{a_{ij}\mid 1\le j\le N_i\}$ be the standard basis of the group $H_i$, $i\ge 1$. 
We define directed elements $\{w_i\mid  i\ge 1\}$ rolling a dice infinitely many times, see \eqref{Yi} and \eqref{wi}. 
Consider the dice group
\begin{equation*}
G:=\langle w_1, A_1 \rangle\subset \Aut T.
\end{equation*}
We assume  that the  dice rolling is lucky at infinitely many steps $i\in\Bbb N$ (i.e. condition {\bf D} or one of its particular stronger versions is valid).
Then $G$ is periodic and elements periods are products of primes from~$P$.
\end{theorem}
\begin{proof}
We start treating the first layer. 
Recall that $w_1$ has the section $w_2$ at $1_{1}\in H_1$. 
Temporarily denote $v_{\a}:=w_1^{\a}=  \a^{-1} w_1 \a\in \St_G(1)$, for all $\a \in H_1$.
Then $(v_\a)^{\b}=v_{\a\b}$, where $\a,\b\in H_1$.
Observe that 
\begin{equation}\label{section}
v_{\a}\big|_{\a^{-1}}=w_2,\qquad \a\in H_1. 
\end{equation}

Consider $g\in G=G_1$. It is a product of  $w_1$'s and elements $A_1\subset H_1$. By moving $H_1$-factors to the left we obtain  a {\it reduced form}
\begin{equation} \label{present2}
g= h v_{\g_1}^{n_1}v_{\g_2}^{n_2}\cdots v_{\g_m}^{n_m},\quad h,\g_l\in H_1,\ 1\le n_j<p_1,\quad m\ge 0,  
\end{equation} 
where the neighbour elements $v_{\g_j}$ are different.
We define the {\it w-length} of presentation~\eqref{present2} counting syllables of conjugates of $w_1$ only: 
$$L(g):= m.$$
We draw attention that $h$ in~\eqref{present} can be equal to the identity element. So, the elements of $H_1$ have zero length
and we count the number of different syllables $v_\g^{n_1}$ above. 

We proceed by induction on $L(g)=m$. 
The base of induction is $m=0$. Then $g=h\in H_1 $ and $h^{p_1}=1$.

{\bf Case 1.} We assume that $h\ne 1$ in presentation~\eqref{present2}. Then we consider
\begin{align} 
g^{p_1}&= h v_{\g_1}^{n_1}\cdots v_{\g_m}^{n_m}\cdot\ldots\cdot     h v_{\g_1}^{n_1}\cdots v_{\g_m}^{n_m}\cdot  h v_{\g_1}^{n_1}\cdots v_{\g_m}^{n_m} \nonumber\\
&= (v_{\g_1}^{n_1})^{h^{p_1-1}}\!\!\!\!\cdots (v_{\g_m}^{n_m})^{h^{p_1-1}}\!\!\!\cdot\ldots\cdot     (v_{\g_1}^{n_1})^h\cdots (v_{\g_m}^{n_m})^h\cdot   v_{\g_1}^{n_1}\cdots v_{\g_m}^{n_m}\nonumber \\
&= v_{\g_1h^{p_1-1}}^{n_1} \cdots v_{\g_mh^{p_1-1}}^{n_m}\cdot \ldots\cdot     v_{\g_1h}^{n_1}\cdots  v_{\g_mh}^{n_m}\cdot   v_{\g_1}^{n_1}\cdots v_{\g_m}^{n_m} \label{F1}\\
&=\big(g_\b\big \rfloor_{\b,\ \b\in H_1}\ \big)\in G_2^{H_1}, \qquad\text{where}\quad g_\b :=g^{p_1}|_{\b}\in G_2=\langle w_2,A_2\rangle.\label{F2}
\end{align}

Consider the action of $h$ on $H_1$ by shifts. We get  $p_1^{N_1-1}$ orbits of length $p_1$. Choose representatives $\{\a_j\mid 1\le j\le p_1^{N_1-1}\}\subset H_1$ in each orbit.
Denote by $m_\a$ the number of syllables $v_\a^{n_\a}$ in~\eqref{present2}, $\a\in H_1$. We get
\begin{equation}\label{Lg}
L(g)=\sum_{j=1}^{p_1^{N_1-1}} (m_{\a_j}+m_{\a_j h}+\cdots+ m_{\a_j h^{p_1-1}})=m.  
\end{equation}

Let $\a\in H_1$, consider the section $g_{\a^{-1}}:=g^{p_1}|_{\a^{-1}} \in G_2$ in~\eqref{F2}. 
By~\eqref{section}  and~\eqref{F1}, all factors $w_2$ in this section are collected from the sections of the $h$-line passing through $\a^{-1}$.   
Now we evaluate the $w$-length of this section that counts now occurrences of the syllables $(w^n_2)^{\b}$, $b\in H_2$ ih the reduced form similar to~\eqref{present2}.
Thus, using~\eqref{Lg} we get for all sections~\eqref{F2}
\begin{equation}\label{Lgm}
\max\{ L(g_{\a_j^{-1}}), L(g_{\a_j^{-1}h}),\ldots, L(g_{\a_j^{-1}h^{p_1-1}})  \}\le m_{\a_j}+m_{\a_jh}+\cdots+m_{\a_jh^{p_1-1}}\le m.  
\end{equation}

a). Not all numbers in~\eqref{Lg} belong to a single $h$-line. Then by bound above $L(g_\a)<m$ for all $\a\in H_1$ and
we apply the inductive assumption to all sections~\eqref{F2}. 

b). All numbers in~\eqref{Lg} belong to a single $h$-line passing through $\a\in H_1$ but  the step $i=1$ is not a lucky roll. 
Then by~\eqref{Lgm} the $w$-lengths of the sections~\eqref{F2} on the line passing through $\a^{-1}$ did no increase,
while sections apart form the line have $w$-lengths equal to zero.
We continue the process, recall that by assumption after finitely many steps we get a lucky roll and we come to the next option.

c). Assume that  all nonzero numbers in~\eqref{Lg} belong to a single $h$-line  $l$  passing through $\a\in H_1$ and $i=1$ is a lucky roll.   
By the shift  $\tilde g:=g^{\a}$, we see that   without loss of generality we can assume that $l$ passes trough $\tilde\a:=1_1\in H_1$.
By assumption {\bf D}, the nontrivial sections $v_\g$ of the record of $g$ as~\eqref{present2} on line $l$ are  only $w_2$ and $\{a_{2,j_1},\ldots, a_{2,j_{S_1}}\}$. 

In order to get a recursive construction we denote $g_1:=g$ and $\Pi_1:=l$.  By assumption c), all $g_1$-sections outside $\Pi_1$ are either trivial or belong to $A_2$. 
We have $g_1^{p_1}\in \St_G(1)\subset G_{2}^{H_1}$.  Let $g_2$ be a section of $g_1^{p_1}$ at $\b\in H_1$. 
Collecting sections in $h$-parallel lines~\eqref{F1}, we observe that
\begin{equation}\label{gg2}
G_{2}\ni g_2:=(g_1)^{p_1}\big |_\b\in 
\begin{cases} H_2, & \b\in H_1\setminus \Pi_1;\\ \langle w_2, a_{2,j_1},\ldots, a_{2,j_{S_1}}\rangle,\quad &\b\in \Pi_1.     
\end{cases} 
\end{equation}
Thus, the sections above outside $\Pi_1$ belong to $H_2$ and disappear while rasing into the power $p_2$.
Thus, we consider the sections $g_2$ above at points of $\Pi_1$ only:
$$
g_2=h_2 g_2',\qquad h_2\in \Pi_2:= \langle a_{2,j_1},\ldots, a_{2,j_{S_1}}\rangle\subset  H_2, \qquad g_2'\in \langle w_2 ^{a_{2,j_1}}\!\!\!,\ldots, w_2^{a_{2,j_{S_1}}}\rangle.
$$
By {\bf D}, while conjugating above the section $w_3$ travels through the $S_1$-dimensional face $\Pi_2\subset H_2$ that passes through $1_2$ parallel to $\{a_{2,j_1},\ldots, a_{2,j_{S_1}}\}$,
while nonempty $g_2$-sections at points of  $H_2\setminus \Pi_2$ belong to $A_3$. 
Thus, the sections of $g_2^{p_2}$ at $\Pi_1$ are obtained  by collecting sections in parallel $h_2$-lines~\eqref{F1}, where $h_2\in\Pi_2$:
\begin{equation}\label{gg3}
G_{3}\ni g_3:=(g_2)^{p_2}\big |_\b\in 
\begin{cases} H_3, & \b\in H_2\setminus \Pi_2;\\ \langle w_3, a_{3,j_1},\ldots, a_{3,j_{S_2}}\rangle,\quad &\b\in \Pi_2.    
\end{cases}
\end{equation}
By assumption, the process terminates at some $m$. Thus we get sections
\begin{equation*}
G_{m +1}\ni g_{m+1}:=(g_m)^{p_m}\big |_\b\in 
\begin{cases} H_{m+1}, & \b\in H_m \setminus \Pi_m;\\ \langle w_{m+1}\rangle,\quad &\b\in \Pi_m.     
\end{cases}
\end{equation*}
This recursion and~\eqref{qqq} yield   $g^{p_1\cdots p_m q}=1$.

{\bf Case 2.} We assume that $h= e$ in presentation~\eqref{present2}. We get
\begin{equation} \label{present3}
g= v_{\g_1}^{n_1}v_{\g_2}^{n_2}\cdots v_{\g_m}^{n_m}
=\big(g_\b\big \rfloor _{\b,\ \b\in H_1} \big), \qquad\text{where}\quad g_\b:=g|_\b \in G_2=\langle w_2,A_2\rangle.
\end{equation}
Since each $v_{\g_i}^{n_i}$ has a unique section $w_2^{n_i}$ that goes into the section $g_{\g_i^{-1}}$ in~\eqref{present3}, we get
$$ 
\sum_{\b\in H_1} L(g_\b)\le L(g)=m.
$$

a). Consider that $L(g_\b)<m$ for all $\b\in H_1$, then we apply the induction to all sections $g_\b$ above.

b). Assume that there exists $\g\in H_1$ such that $L(g_\g)=m=L(g)$.
Since $w_2$ appears in the section $g_\g$ only for $v_{\g^{-1}}$ it follows that $m=1$ and $g=v_{\g^{-1}}^{n}$, hence $g^{q}=1$. 
\end{proof} 

\begin{corollary}
Assume that the construction of the dice group is periodic, i.e. there exists $M\ge 1$ such that $H_{i+M}\cong H_i$ for all $i\ge 1$,
and the same periodicity is valid for  dice rolling as well. Assume that we have at least one lucky roll in the period.
Then the dice group $G$ is self-similar and periodic.
\end{corollary}
\begin{proof} Self-similarity follows from~\eqref{selfN}.
\end{proof}

\begin{remark}
Observe that we can easily make dice groups infinite, periodic (and self-similar if required). 
\begin{itemize}
\item
We need to avoid only the group $\mathbf C_2$. Even using sometimes $\mathbf C_3$ we can get an infinite periodic (and self-similar) group with oscillating sizes of the groups $H_n$.
\item
For example, we start with $G:=G_1:=\langle w_1, a_{11}\rangle $, where $a:=a_{11}$ is the generator of $\mathbf C_3$ and 
$w_1:=(w_2\rfloor_{1_1},a_{21}\rfloor_a, a_{22}\rfloor_{a^2})$, and $A_2:=\{a_{21}, a_{22}\}$ is the basis of $H_2:=\mathbf C_3^2$, 
so the next group is $G_2:=\langle w_2, A_2\rangle $ and we can continue the process.
In a similar way we can include large pieces of the Gupta-Fabrykowski group~\cite{FabGup85} and obtain periodic (optionally self-similar) groups.
\item  
If we want to have a lucky dice roll at step $i$, then case $H_{i+1}:=\mathbf C_2^2$ is not acceptable. 
\item
On the other hand we can set  $H_{n}:=\mathbf C_3^2$  for all $n\ge 1$ and even satisfy {\bf DDmin} at all steps.
\end{itemize} 
\end{remark}


\end{document}